\documentclass[10pt]{amsart}
\usepackage{setspace}
\usepackage[all,cmtip]{xy}
\usepackage{hyperref}

\usepackage{amsmath}
\usepackage{amsthm}
\usepackage{amsfonts}
\usepackage{amssymb}
\usepackage{graphicx}
\usepackage{dsfont}
\usepackage{pifont}
\usepackage{color}
\usepackage{verbatim}

\theoremstyle{definition}
\newtheorem{theorem}{Theorem}[section]
\newtheorem{proposition}[theorem]{Proposition}
\newtheorem{lemma}[theorem]{Lemma}
\newtheorem{definition}[theorem]{Definition}
\newtheorem{remark}[theorem]{Remark}
\theoremstyle{remark}
\newtheorem{example}[theorem]{Example}

\newcommand{\spec}[0]{\ensuremath{\text{Spec }}}
\newcommand{\half}[0]{\frac{1}{2}}
\newcommand{\wt}[1]{\ensuremath{\widetilde{#1}}}
\newcommand{\abs}[1]{\left \vert #1 \right\vert}
\newcommand{\takes}[2]{\!:\!#1 \rightarrow #2}
\newcommand{\ra}[0]{\ensuremath{\rightarrow}}
\newcommand{\lra}[0]{\ensuremath{\longrightarrow}}
\newcommand{\nlra}[1]{\stackrel{#1}{\lra}}

\newcommand{\PP}[0]{\ensuremath{\mathds{P}}}
\newcommand{\ZZ}[0]{\ensuremath{\mathds{Z}}}
\newcommand{\RR}[0]{\ensuremath{\mathds{R}}}
\newcommand{\QQ}[0]{\ensuremath{\mathds{Q}}}
\newcommand{\CC}[0]{\ensuremath{\mathds{C}}}

\begin{document}
\bibliographystyle{halpha}
\onehalfspacing

\title[On the existence of affine LG phases in GLSMs]{On the existence of
affine Landau-Ginzburg phases in gauged linear sigma models}

\author[Patrick Clarke]{Patrick Clarke${}^\dagger$}
\address{${}^\dagger$Department of Mathematics, University of Pennsylvania}
\email{pclarke@math.upenn.edu}

\author[Josh Guffin]{Josh Guffin${}^\ddagger$}
\address{${}^\ddagger$Department of Mathematics, University of Pennsylvania}
\email{guffin@math.upenn.edu}

\begin{abstract}
  We prove a simple criterion for the existence of an affine Landau-Ginzburg 
  point in the K\"ahler moduli space of a gauged linear sigma model. 
\end{abstract}
\maketitle

\section{Introduction}

A Landau-Ginzburg (LG) point in the K\"ahler moduli space of a gauged
linear sigma model (GLSM) allows one to compute correlation functions which
are otherwise inaccessible.  In the literature, there is no efficient
method for determining the existence of such a point.  \\

GLSMs were introduced in \cite{Witten:1993yc} as a way of studying
correlation functions that compute Gromov-Witten invariants in certain
related non-linear sigma models (NLSM).  In particular, the low-energy
limit on the Higgs branch of a GLSM is a NLSM whose target space is a toric
variety $X$ obtained as a symplectic $U(1)^\rho$ quotient using the
parameters of the GLSM.\\

For certain classes of GLSMs, one may choose the parameters so that the
target space of the low-energy theory is an orbifold $\CC^n/\Gamma$, with a
superpotential whose critical locus lies at the fixed point of the finite
abelian group $\Gamma$.  Such a setup is called a Landau-Ginzburg theory,
and correlators therein are exactly soluble.  To find if a given GLSM
possesses an LG point, one typically constructs the secondary fan and
laboriously checks whether the D- and F-terms for each chamber satisfy the
requisite condition.  Such a strategy was pursued in \S4.2 of
\cite{Morrison:1994fr}, for example. \\

LG points are also useful in $(0,2)$ GLSMs.  These models reduce in the
low-energy limit to a NLSM on the same variety $X$, but also depend on a
holomorphic vector bundle $\mathcal V \ra X$ that is determined by the data of
the $(0,2)$ GLSM.  Correlation functions in these theories are invariants of
$(X,\mathcal V)$ that generalize the Gromov-Witten invariants of $(X, T_X)$
\cite{Adams:2005tc}.  In cases where the bundle $\mathcal V$ is a deformation
of the tangent bundle, correlators may be computed using a brute-force method
based on \v Cech cohomology \cite{Katz:2004nn,Guffin:2007mp}, or by employing
more refined techniques in the $(0,2)$ GLSM \cite{McOrist:2007kp}.
When $\mathcal V$ is not a deformation of $T_X$, general techniques to
compute correlators in the NLSM do not exist.  However, when the $(0,2)$
GLSM admits an LG phase, correlators may be computed using the methods of
\cite{Melnikov:2009nh}.  Our results should admit a generalization to the
$(0,2)$ case.\\

LG theories have also been useful for the computation of Gromov-Witten
invariants, following the program initiated by Fan, Jarvis, and
Ruan\cite{Fan:2007vf,Fan:2007ms}.  Our results should be especially useful
in this setting for producing previously unknown classes of LG models for
study (see Remark
\ref{rmk:produce}).

\subsection{The Criterion}

Recently, Herbst conjectured \cite{Herbst:2009pv} a criterion for the
existence of a Landau-Ginzburg point in terms of the charge matrix of the
GLSM.  The Herbst criterion can be slightly simplified, and we prove that
the simplified version is equivalent to the existence of an affine LG point
provided the critical locus of the superpotential is compact for some value
(and thus all values) of the K\"ahler parameter.  To this end, we also
provide a rigorous definition of an affine LG point in a GLSM, and prove the
equivalence of symplectic and algebraic quotients for arbitrary GLSMs
without regard for smoothness or compactness.  \\

The {\it Herbst Criterion} proven herein takes the following form.
Consider the charge matrix of the GLSM, $Q$: if the rank of the GLSM gauge
group is $\rho$ and there are $N$ chiral fields, then $Q$ is a $\rho \times
N$ matrix of rank $r$.  Then an affine LG point exists whenever one can choose $r$
linearly-independent columns so that the other $n:=N-r$ columns lie in the
negative cone of the chosen $r$.  This setup is a slight generalization of
the charge matrices normally considered in the physics literature, where
$Q$ is assumed to be full rank.  We provide a mathematical setting for $Q$
in our discussion of symplectic quotients -- see equation
\eqref{eq:torusseq}.  \\

Herbst originally included the condition that the chosen columns should be
unique, in the sense that no column amongst the remaining $n$ is a copy of
one of the chosen.  However, this is implied by the condition that all
others lie in the negative cone.  On the other hand, this corollary can be
quite useful in showing by hand that a given model does not have an affine LG
point.  A precise statement of the Herbst Criterion is given in Definition
\ref{def:hc}.\\

Heretofore, the main class of toric varieties known to admit affine LG phases
were the total spaces of the canonical bundle over compact toric varieties.
Our analysis shows that LG phases are extremely common, and easily produced.
Indeed, our results provide a way to produce \emph{every} possible affine LG
phase for any GLSM -- see Remark \ref{rmk:produce}.

\subsection{Examples}

Before proceeding, we present three well-known examples and discuss the
application of the criterion to them.

\begin{example}
  Consider the canonical bundle of $\PP^m$.  Here, the gauge group is rank
  1 and the charges are arranged in a $1\times (m+2)$ matrix normally
  written as
  \[
  Q = \left (
  \begin{matrix}
	 1 & 1 & \cdots & 1 & -m-1 \\
  \end{matrix}
  \right ).
  \]
  Discarding the first $m+1$ columns due to multiplicity, one is left with
  the column $(-m-1)$.  Since the first $m+1$ columns are in its negative
  cone, this model will possess an LG point for an appropriate choice of
  superpotential.
  \label{ex:canonicalpm}
\end{example}

\begin{example}
  For the canonical line bundle over the product of rational curves, $K \ra
  \PP^1 \times \PP^1$, the gauge group is rank 2 and the charge matrix is
  normally written 
  \[
  Q = \left (
  \begin{matrix}
	 1 & 1 & 0 & 0 & -2 \\
	 0 & 0 & 1 & 1 & -2
  \end{matrix}
  \right ).
  \]
  Examining this matrix, one discards the first four columns since each
  vector occurs with multiplicity and finds that there is only one
  independent unique column.  Therefore this model cannot have an affine LG point.
\end{example}

\begin{example}
  \label{example:canonicalrwp}
  Consider the canonical bundle of the resolved weighted projective space
  $\wt\PP^4_{1,1,2,2,2}$.  This model has a rank 2 gauge group and charge
  matrix
  \begin{equation}
  Q = \left (
  \begin{matrix}
	 0 & 0 & 1 & 1 & 1 & 1  & -4 \\
	 1 & 1 & 0 & 0 & 0 & -2 &  0 
  \end{matrix}
  \right ).
  \label{eq:rwp4cm}
  \end{equation}
  After discarding the first five columns one is left with the final two, which
  are linearly independent.  It is not hard to check that the first five
  columns are contained in the negative cone of the last two.  One can also use
  the algorithm outlined in Remark \ref{rmk:algorithm} to find that the charge
  matrix row-reduces to 
  \[
	 \left(
	 \begin{matrix}
		1 & -4 \\
		-2 & 0 
	 \end{matrix}
	 \right)^{-1}
	 \cdot Q = 
	 \left(
	 \begin{matrix}
		-\frac 1  2  & -\frac 1  2  & 0 & 0 & 0 & 1 & 0 \\
		-\frac 1  8  & -\frac 1  8  & -\frac 1  4  & -\frac 1  4  & -\frac 1  4 
		& 0 & 1
	 \end{matrix}
	 \right),
  \]
  so that the first five vectors clearly lie in the negative cone. Thus,
  for an appropriate choice of superpotential this model will possess an LG
  point.
\end{example}

\subsection*{Acknowledgments}

We thank Manfred Herbst for helpful comments on an early draft of this
manuscript, Jacques Distler for useful conversations, and the organizers of the
2009 conference on \emph{(0,2) Mirror Symmetry and Quantum Sheaf Cohomology} at the
Max-Planck-Institut f\"ur Gravitationsphysik where this work began.  This
material is based upon work supported by the National Science Foundation under
DMS Grant No. 0636606 and 0703643.

\section{Physical Exposition}

We first explain the physical origins of LG points in the GLSM.  We will always
assume that we are dealing with a GLSM featuring a gauge group of rank $\rho$
and $N$ chiral bosons, whose low-energy theory describes a toric variety $X$
for appropriate values of its K\"ahler parameters. \\

To determine the low-energy theory, one imposes a system of constraints arising
as the classical equations of motion of the GLSM Lagrangian on the chiral
bosons in the theory.  In terms of the charge matrix $Q$, chiral bosons
$\phi^i$, and K\"ahler parameters $r^a$, one has $r$ equations
\begin{equation}
  \sum_{i=1}^N Q_i^a \abs{\phi^i}^2 - r^a = 0 \qquad \text{ for } 1 \leq a \leq r.
  \label{eq:dterms}
\end{equation}
Here $r$ is the rank of $Q$.  For the purposes of this paper, we will
simply call the equations above \emph{D-terms}.\\

The D-terms specify the construction of a toric variety $X$ as a
$U(1)^\rho$ quotient, as in equation \eqref{eq:u1quotient}.  The low-energy
theory is also dependent on another class of terms arising from a
torus-invariant holomorphic function $W\takes {\CC^N} \CC$ called the
\emph{superpotential}.  Several terms in the Lagrangian involve this
function, but we will concentrate on one set in particular whose vanishing
is required for supersymmetric vacua:
\begin{equation}
  \sum_{i=1}^N \abs{\frac{\partial W}{\partial \phi^i}}^2 = 0.
  \label{eq:fterms}
\end{equation}
For the purposes of this paper, these will be called \emph{F-terms}.\\

Let $\mathcal V \ra Y$ be a vector bundle of rank $k$ over a compact toric
variety $Y$, whose total space is the $n$-dimensional toric variety $X$.
Let $Z \subset Y$ be the smooth vanishing locus of a holomorphic section
$f\takes Y {\mathcal V}^\vee$.  For certain values of the K\"ahler
parameters -- those in the K\"ahler cone -- such a compact complete
intersection is realized as the target space in the low-energy theory of a
GLSM whose superpotential is
\[
W = \sum_{\alpha = 1}^k p_\alpha f^\alpha(\phi),
\]
where $p_\alpha$ are a subset of the bosonic fields associated with fiber
coordinates on $\mathcal V$ and the $\phi$ are fields associated with the base
$Y$.  Since the zero locus of $f$ is assumed to be smooth, the critical points
of $W$ are precisely the vanishing locus of $f$, lying within the zero section
of the bundle.\\

The model possesses an LG point if there is some value of the K\"ahler
parameters (taken to lie deep in the interior of a top-dimensional cone of the
secondary fan) such that solving the D- and F-terms requires that precisely $r$
of the bosons get a vacuum expectation value (VEV), while the others vanish.
Additionally, one requires that the low-energy superpotential in this phase has
a single degenerate critical point.  
In other words, an LG point is a choice of K\"ahler parameters such that the
low-energy physics is described by a quantum field theory whose bosonic fields
are valued in a vector space (in particular $\CC^n$, or more generally its
quotient by a finite abelian group), governed by a superpotential, and whose
space of vacua consists of a single point. One must take the parameters to be
deep inside a cone to avoid subtleties from quantum corrections.\\

\section{Mathematical Exposition}

Dividing out the $U(1)^\rho$ symmetries of a GLSM with specified D-terms leads
to a toric variety which, along with the superpotential, governs its physics.
The toric variety is obtained as a symplectic quotient at an appropriate value
of the moment map (the $r^a$ in \ref{eq:dterms}).  \\

It is well-known that if the value of the moment map is regular and the
quotient is compact, then it is a projective toric manifold
\cite{Audin:2004ta}.  It is not hard to imagine, though less well-known, that
the quotient is a quasi-projective variety at any value of the moment map,
regardless of regularity or compactness.  We provide a proof of this fact in
\S\ref{ssec:quot}.\\

The group of characters on the quotient, $M \cong \ZZ^{\oplus n}$, fits into an
exact sequence 
\begin{equation}
0 \ra  M \nlra A (\mathfrak u(1)_{\ZZ}^N)^* \nlra 
Q (\mathfrak u(1)_{\ZZ}^\rho)^*.
  \label{eq:torusseq}
\end{equation}
Here $\mathfrak u(1)^N$ is the Lie algebra of a maximal torus in
$\text{Aut}(\CC^N)$ commuting with the action of $\mathfrak u(1)^\rho$ and
$\mathfrak u(1)_\ZZ$ denotes the kernel of the exponential map.  The
$\RR$-linear maps obtained from these by $\otimes_\ZZ \RR$ are denoted by the
same name.  \\

As in equation \eqref{eq:linearaction} below, $Q^t$ specifies the a linear
action of $(\CC^*)^\rho$ on $\CC^N$ and $A$ is the kernel of $Q$.  Physically,
$Q$ is the matrix of charges appearing in the D-terms \eqref{eq:dterms} of the
GLSM.  The Lie algebra $\mathfrak u(1)^N$ has a canonical set, $C = \{
\partial_{\psi_1} \cdots \partial_{\psi_{N}} \}$, corresponding to coordinates
$z_j$ with $\arg{z_j} = 2 \pi \psi_j$.

\begin{definition}
  The map $Q$ in equation \eqref{eq:torusseq} satisfies the \emph{Herbst Criterion}
  if there exists a subset $\{ h_1, \cdots, h_r \}$ of  
  $\{ \partial_{\psi_1}, \cdots , \partial_{\psi_N} \}$ such that
  $\{ Q h_1, \cdots, Q h_r \}$ are linearly independent and 
  \[Q C \setminus \{ Q h_1, \cdots, Q h_r \} \subset 
  \RR_{\leq 0} Q h_1 + \cdots  + \RR_{\leq 0} Q h_r.\]
  \label{def:hc}
\end{definition}

In particular, isomorphisms $(\mathfrak u(1)^\rho)^*\cong \RR^\rho = \ZZ^\rho
\otimes_\ZZ \RR$ and $(\mathfrak u(1)^N)^*\cong \RR^N = \ZZ^N \otimes_\ZZ \RR$
give an integer matrix expression for $Q$.  The criterion is satisfied if there
exists an reordering of the basis of $\RR^N$ and a rational change of basis for
$\RR^\rho$ so that 
\begin{equation}
  Q = 
  \left(
  \begin{matrix}
    \mathds 1_{r\times r} & \mathbf n_{r \times n} \\
    0 & 0
  \end{matrix}
  \right),
  \label{eq:hc}
\end{equation}
with the entries of $\mathbf n$ non-positive rational numbers.  We will prove
the equivalence of the Herbst Criterion to the existence of an affine quotient
in \S\ref{ssec:proofHerbst}. 

\subsection{Toric Quotients}
\label{ssec:quot}

The most important consequence of the equality of symplectic and algebraic
quotients is that we can read off the algebro-geometric description of the
quotient from its image under the moment map.  This image, known as the moment
polyhedral set, is defined by a finite family of inequalities. \\

Later, we rely on the geometry of the polyhedral set to understand when
quotient is unchanged for small changes in the K\"ahler parameter.\\

Every linear action of an algebraic torus on an $N$-dimensional complex vector
space $V$,
\(
(\CC^*)^\rho 
\circlearrowright V,
\)
may be unitarily diagonalized so that for any $\vec \lambda \in
(\CC^*)^\rho$, $\vec z \in \CC^N$, and integers $Q^t_{ij}$ for $1 \leq i
\leq N, 1\leq j \leq \rho$,
\begin{equation}
  \vec{\lambda} \star \vec{z} = 
  (\lambda_1^{Q^t_{11}} \cdots \lambda_\rho^{Q^t_{1\rho}} z_1,\cdots,  
  \lambda_1^{Q^t_{N1}} \cdots \lambda_\rho^{Q^t_{N\rho}} z_N).
\label{eq:linearaction}
\end{equation}
The standard symplectic form $\omega = \frac{\sqrt{-1}}{2} \sum_i dz^i
\wedge d\overline{z}^i$ and the standard action of $U(1)^N$ on $\CC^N $
define a moment map 
\begin{equation*}
  \mu \takes{\CC^N}{\RR^N} = (\mathfrak{u}(1)^N)^*  .
\end{equation*}
In coordinates, $\mu$ is given by $\vec{z} \mapsto \half (|z_i|^2)_i$.  \\

The moment map for the action of $U(1)^\rho \subset (\CC^*)^\rho$ on
$(\CC^N, \omega)$ is given by the composition
\begin{equation*}
  \mu_Q = Q \circ \mu \takes{\CC^N}{\RR^\rho} = (\mathfrak{u}(1)^\rho)^*.
\end{equation*}
Because of D-term contributions \eqref{eq:dterms} to the Lagrangian, the
physical action is stationary for maps $\Sigma \rightarrow \mu_Q^{-1}(s)$
for a choice of $s \in \RR^\rho$ .  If $s$ is not in the image of $Q$, $X$
is empty.  Otherwise, $s$ corresponds to $r^a$ in the D-terms.  These maps
are taken up to the action of $U(1)^\rho$ on $\mu_Q^{-1}(s)$.  Thus,
it is equivalent to consider maps $\Sigma \rightarrow X$, where 
\begin{equation}
  X = \mu_Q^{-1}(s) / U(1)^\rho.
  \label{eq:u1quotient}
\end{equation}

$X$ naturally carries the structure of a toric variety and can be written
as a geometric quotient as follows.  Define $E$ to be the complement of
$(\CC^*)^N \star \mu_Q^{-1}(s)$ in $\CC^N$, and consider 
\begin{equation*}
  \mathds{X} = (\CC^N \setminus E) / (\CC^*)^\rho .
\end{equation*}
Since the $U(1)^\rho$ orbits are contained in the $(\CC^*)^\rho$ orbits,
there is a natural map $X \rightarrow \mathds{X}$.  We will show that this
map is an isomorphism.  We first check that it is an injection by showing
that the orbits  $\RR_+^r \star \vec{z}$ for $\vec{z} \in \mu_Q^{-1}(s)$
are disjoint.  This is accomplished by showing that $\mu_Q$ restricted to
such an orbit is injective.  

\begin{lemma}
  $X \rightarrow \mathds{X}$ is injective
  \label{lem:inj}
\end{lemma}

\begin{proof}
  The action of $\RR_+^\rho \subset (\CC^*)^\rho$ on $\CC^N$ induces an
  action of $\RR_+^\rho$ on $\RR_{\geq 0}^N$ defined by $\vec{\lambda}
  \star \mu(\vec{z}) = \mu(\vec{\lambda} \star \vec{z})$.  It is easy to
  see that this is independent of the choice of $\vec{z}$, as different
  choices are given by the action of $U(1)^N$.  Directly, the action
  $\RR_+^\rho$ on $\RR_{\geq 0}^N$ is given by 
  \begin{equation}
    \vec{\lambda} \star \vec{q} = \vec{\lambda}^{Q^t} \vec{q} = \text{diag}(\lambda^{{Q^t}_1}, \cdots \lambda^{{Q^t}_N}) \vec{q}
    \label{eq:RpAction}
  \end{equation}
  where $\lambda^{{Q^t}_i} = \lambda_1^{{Q^t}_{i1}} \cdots
  \lambda_\rho^{{Q^t}_{i\rho}}$.  \\
  
  We would like to show that the action of $\RR_+^r$ on $\CC^N$ changes the
  value of $\mu_Q$.  It suffices to show that $\RR_+^r$ on $\RR_{\geq 0}^N$
  changes the value of $q$.  Concretely, we wish to show $Q (\mathds{1} -
  \vec{\lambda}^{ Q^t}) \vec{q} = 0$ has no solutions except $\vec{q} = 0$.\\
  
  If we denote the kernel of $Q$ by $A$ as in equation \eqref{eq:torusseq},
  then it is equivalent to show
  \begin{equation*}
    (\mathds{1} - \vec{\lambda}^{ Q^t})
    \RR_{\geq 0}^N \cap \text{Im} 
    A = \{ 0 \} .
  \end{equation*}
  
  Because both $1-x$ and $-\log(x)$ are positive/negative/zero on the same
  set we have
  \begin{equation*}
    (\mathds{1} - \vec{\lambda}^{Q^t})
    \RR_{\geq 0}^N 
    = -\log(\vec{\lambda}^{Q^t}) \RR_{\geq 0}^N 
    = -\text{diag}({Q^t} \log(\vec{\lambda})) \RR_{\geq 0}^N .
  \end{equation*}
  An element of $-\text{diag}({Q^t} \log(\vec{\lambda})) \RR_{\geq 0}^N$ is
  non-zero if and only if its dot product the vector
  \begin{equation*}
    -\text{diag}({Q^t} \log(\vec{\lambda})) 
    \left[
      \begin{array}{c}
	1 \\
	\vdots \\
	1
      \end{array}
      \right] = -{Q^t} \log(\vec{\lambda})
  \end{equation*}
  is non-zero.  Finally, observe that the dot product for any $v \in M_\RR$
  is given by 
  \begin{equation*}
    -{Q^t} \log(\vec{\lambda}) \cdot A v = 
    -\log(\vec{\lambda})^t 
    Q  A v = 0.
  \end{equation*} 
\end{proof}


\begin{lemma}
  $X \rightarrow \mathds{X}$ is surjective.
\end{lemma}

\begin{proof}
  First notice that Lemma \ref{lem:inj} guarantees 
  \( X = ((\CC^*)^\rho \star
  \mu_Q^{-1}(s))/(\mathds{C^*})^\rho.\)
  So, we need to check that 
  \begin{equation*}
	 (\CC^*)^N \star \mu_Q^{-1}(s)
	 = (\CC^*)^\rho \star \mu_Q^{-1}(s) .
  \end{equation*}
  We will show by construction that given any element $p \in (\CC^*)^N
  \star \mu_Q^{-1}(s)$ there exists an element $h \in (\CC^*)^\rho$ such that
  $\mu_Q(h \star p) =  s$.  Choosing $g$ such that $\mu_Q(g \star p) = s$, $h$
  will be constructed by lifting the curve $\gamma(t) = \mu_Q(\exp(t \ln{g})
  \star p) \subset \RR^{\rho}$ to a curve in $(\CC^*)^\rho$.\\

  Denote the Lie algebra of $(\CC^*)^N$ by $\mathfrak{g}$ and the Lie algebra
  of $(\CC^*)^\rho$ by $\mathfrak{h}$.  Then, given a curve $\eta\takes
  {[0,1]_t}{\mathfrak{g}}$, we obtain a curve 
  \( P\exp(\eta)\takes{[0,1]_t}{(\CC^*)^N}\)
  defined by $\frac{d}{dt} P\exp(\eta) = \eta$.  Furthermore if $\eta$ lies in
  $\mathfrak{h}$ then the resulting curve is in $(\CC^*)^\rho$.  To be clear,
  this notation is with respect to the trivialization by right-invariant vector
  fields: $T_e(\CC^*)^N = \mathfrak{g}$.\\

  The typical fibre of the tangent bundle of $(\CC^*)^N \times \CC^N$ is
  $\mathfrak{g} \oplus \CC^N$.  Consider the differential $d (\mu_Q \circ
  \star)$ of the composition of $\star$ with $\mu_Q \takes{\CC^N}
  {\RR^\rho}$.  Restricting to $\mathfrak{g}$ induces a map from Lie
  algebra of $(\CC^*)^N$ to the tangent space of $\RR^\rho$.  Given a curve
  $\gamma\takes{[0,1]} {\RR^\rho}$ such that $\gamma'(t)$ is in the image
  of $\mathfrak{g}$ under $d(\mu_Q \circ \star)$, denote the lifted curve by
  $\hat{\gamma}\takes{[0,1]}{\mathfrak{g}}$.  Observe that it has
  the property that $\mu_Q(P\exp(\hat{\gamma}(t)) \star l_0) = \gamma(t)$ for
  any choice $l_0 \in \mu_Q^{-1}(\gamma(0))$.\\
  
  Now consider the curve $\gamma(t) = \mu_Q(\exp(t \ln{g}) \star p)$.  Observe
  that $\gamma(0) = \mu_Q(p)$ and $\gamma(1) = s$.  The proof then depends on the
  existence of a lift of $\gamma'$, $\eta\takes{[0,1]}{\mathfrak{h}}
  \subset \mathfrak{g}$.  Once we have this lift, we can take  $h = P
  \exp(\eta(1))$.  By construction, $\gamma' \subset d(\mu_Q \circ
  \star)(\mathfrak{g})$ since $\hat{\gamma}(t) = t \ln{g}$ is a lift to
  $\mathfrak{g}$.  We will prove the existence of $\eta$ by showing $d(\mu_Q
  \circ \star)(\mathfrak{g}) = d(\mu_Q \circ \star)(\mathfrak{h})$.  \\

  The differential $d(\mu_Q \circ \star)$ annihilates $\mathfrak{u}(1)^N$.
  So, we can restrict our attention to its evaluation on $\RR_+^N$.  As before,
  the action is diagonal and so the rank of the differential at $\vec{q} \in
  \RR_{\geq 0}$ equals the dimension of the smallest coordinate subspace
  containing $\vec{q}$ -- that is, the number of non-zero entries of $\vec{q}$.
  Note that it suffices to check when $\vec{q}$ has all non-zero entries, as
  the appearance of a zero-entry is the same as replacing $N$ with $N-1$.\\

  As in equation \eqref{eq:RpAction}, the action of $\RR_+^\rho$ is given by
  the product $\vec{\lambda}^{Q^t} \vec{q}$.  The differential of the action is  
  \begin{equation*}
    d (\vec{\lambda}^{Q^t}) \cdot \vec{q} =
    \text{diag}(\sum_{j=1}^\rho \lambda^{Q^t_i} \ Q^t_{ij} \  d\log{\lambda_j})_i \cdot \vec{q}. 
  \end{equation*}
  The Jacobian can be written $\text{diag}(q_i \lambda^{Q^t_i}) Q^t$, and it
  follows that the Jacobian of $\mu_Q \circ \star$ is given by $Q\,
  \text{diag}(q_i \lambda^{Q^t_i}) Q^t$.  If we set $f^2 = \text{diag}(q_i
  \lambda_i^{Q^t})$ for a diagonal square matrix $f$, then we can write the
  Jacobian as 
  \begin{equation*}
    (fQ^t)^t (f Q^t).
  \end{equation*}
  The rank of $f Q^t$ is the rank of $Q^t$, and it is not too difficult to
  check that any matrix of the form $L^t L$ has the same rank as $L$.  It
  is also easy to check that the rank of $d(\mu_Q \circ \star)$ also equals
  the rank of $Q^t$, so we are done.
\end{proof}

One consequence of these proofs is that if
$\text{Stab}_{U(1)^\rho}(\vec{z})$ is isomorphic to $T \times U(1)^\ell$,
for some torsion group $T$, then $\text{Stab}_{(\CC^*)^\rho}(\vec{\lambda}
\star \vec{z})$ is isomorphic to $T \times (\CC^*)^\ell$ for any
$\vec{\lambda} \in (\CC^*)^\rho$.  \\

As with usual symplectic quotients (no regularity assumption here), the map
$\mu$ restricted to $\mu_Q^{-1}(s)$ descends to a map $\mu_X\takes X
{\RR^N}$.  $\mu_X(X)$ lies in the affine translation of $\text{Im
}A$ over $s$, so we regard it as a map 
\begin{equation*}
  \mu_X \takes X {\text{Im }A} \cong M_\mathds{R}.
\end{equation*}
Here $M_\RR = M \otimes_\ZZ \RR$.

\subsection{The Herbst Criterion and affine Landau-Ginzburg points}
\label{ssec:proofHerbst}

The vector space $\text{im } Q \cong \RR^r \subset (\mathfrak u(1)^\rho)^*$
may be decomposed into regions fitting together as a polyhedral fan known
as the \emph{secondary fan}\cite{Billera:1990sp}. The level sets of $\mu_Q$
over the relative interior of its top-dimensional cones define isomorphic
toric varieties.\\

Furthermore, given two adjacent top-dimensional cones and corresponding
varieties, the quotient construction in equation \eqref{eq:u1quotient}
applied to the codimension-one cone separating them induces a proper
birational transformation between the two varieties.  See \S3.4 of
\cite{Cox:2000vi} for a nice exposition.\\

Given a point $s \in \RR^r$, one constructs the corresponding polytope by
first choosing a lift $\wt s \in \RR^N $ such that $Q(\wt s) = s$.  The
polytope is defined to be 
\begin{equation}
  P_s := \left \{ m \in M_\RR \; | \; A(m) + \wt s \geq 0 \right \}.
  \label{eq:polytope}
\end{equation}
The constraints in the definition of $P_s$ are easily reinterpreted as a
collection of $N$ half-spaces in $M_\RR$ whose inward normal vectors are
the rows of $A$.  Furthermore, one can show that the image of the polytope
under the map $m \mapsto A(m) + \wt s$ is exactly $\mu_X(X)$.  For this
reason, $P_s$ is known as the moment polyhedral set of $X$ at level $s$.

\begin{definition}
  The relative interior of a top-dimensional cone of the secondary fan is
  called a \emph{phase} of the associated GLSM.
\end{definition}

\begin{definition}
  A point in the secondary fan is \emph{stable} if it is contained in
  the relative interior of a top-dimensional cone.
\end{definition}

\begin{definition}
  A stable point $s \in \RR^r$ is \emph{affine} if
  the polytope $P_s$ is a top-dimensional simplicial cone in $M_\RR$.
\end{definition}

The nomenclature \emph{affine} is justified, since via standard
construction \cite{Fulton:1993tv}, vertices in the polytope correspond to affine open
sets that are glued together using the data of higher dimension faces.
As there is one vertex in a polyhedral cone there is only one open set.
In \S\ref{ssec:orbifold}, we will show that in fact the quotient is
$\CC^n / \Gamma$ for $\Gamma$ a finite abelian group.  

\begin{lemma}
  \label{lem:hcasp}
  The map $Q$ satisfies the Herbst Criterion iff there is an affine stable
  point in its image.
\end{lemma}

\begin{proof}
  Consider a charge matrix $Q$ satisfying the Herbst Criterion -- it may be
  written in the form \eqref{eq:hc}.  Since $A$ is full rank, there is a
  basis such that
  \begin{equation}
	 A = \left ( \begin{matrix}
		\mathbf N_{r \times n} \\
		\mathds 1_{n \times n}
	 \end{matrix} \right).
  \label{eq:div}
  \end{equation}
  Furthermore, the sequence \eqref{eq:torusseq} implies
  that $ \mathbf n_{r \times n } = - \mathbf N_{r \times n}$. \\

  We now use these facts to construct a simplicial polytope.  Select a
  point $\sigma$ in the positive orthant of $\RR^r$ (positive in the basis
  chosen so that $Q$ is as in equation \eqref{eq:hc}).  One may then select a
  lift $\wt \sigma\in \RR^N$ whose first $r$ entries are $\sigma$ and whose
  final $n$ entries are zero. \\

  The resulting half-spaces are of two types: those arising the first $r$
  rows of $A$, and those arising from the last $n$ rows.  Those defined by
  the first $r$ rows are of the form
  \begin{equation}
	 H^+_i = \{ m \in M_\RR \; |\;  \mathbf N_{i}\cdot m \geq - \sigma_i\},
	 \label{eq:polytopeNs}
  \end{equation}
  while those defined by the last $n$ rows take the form
  \[
  H^+_j = \{ m \in M_\RR \; |\;  m_j \geq 0\}.
  \]
  Half-spaces defined by the last rows pick out the positive orthant of
  $M_\RR$, and since the entries of $\mathbf N$ are all non-negative, any
  $m\in M_\RR^+$ will satisfy the inequality \eqref{eq:polytopeNs}.  Thus,
  the polytope consists of the positive orthant, which is a top-dimensional
  simplicial cone in $M_\RR$.  Since $\sigma$ may be taken to be any
  element in the positive orthant, one may choose it to be stable. \\

  To prove the converse, let $P \subset M_\RR \cong \RR^n$ be a
  top-dimensional simplicial cone defined by $N$ half-spaces.  We first
  show by contradiction that the inward-pointing normal vectors of all $N$
  half-spaces are contained in $P$. \\

  By translations, we can take the apex of the cone $P$ to be the origin in
  $M_\RR$.  Let $\{e_1,\cdots, e_n\}$ generate the rays of $P$, so that $P
  = \text{Cone}(e_1, \cdots, e_n)$.  Define $\langle \ , \  \rangle$ so
  these vectors are orthonormal.  Let $H^+$ be a half-space with normal
  $\zeta \not \in P$:  by definition, we have 
  \[
  H^+ := \{ m \in M_\RR \; | \; \langle m, \zeta \rangle \geq - a \}.
  \]
  Note that if $a < 0$, $H^+$ does not contain the origin, so that $a \geq
  0$. Furthermore, stability implies that $a > 0$.\\
  
  Since $\zeta \not \in \text{Cone}(e_1, \cdots, e_n)$, there exists an
  $e_j$ with $\langle e_j, \zeta \rangle = \beta < 0$. Then for all $\alpha
  > \abs{a\slash\beta}$, 
  \[
  \langle \alpha e_j, \zeta \rangle  = \alpha \beta < - \abs a.
  \]
  Therefore $\alpha e_j$ is not in $\text{Cone}(e_1,\cdots, e_n)$, a
  contradiction.\\

  Since all the inward normals are positive,  a basis for $A$ exists (the
  $e_j$) such that \eqref{eq:div} holds with all entries of $\mathbf N$
  positive -- exactly the Herbst Criterion.  
\end{proof}

\begin{definition}
  An \emph{affine Landau-Ginzburg point} of a GLSM is an affine stable point such
  that $W\takes {X} \CC$ has an isolated critical point.  The phase in
  which an affine LG point lies is known as an \emph{affine Landau-Ginzburg phase}.
  \label{def:LG}
\end{definition}

Hereafter, we will refer to an affine LG point (phase) as an LG point
(phase).
As mentioned earlier, in an LG phase
$X \cong \CC^n / \Gamma$ for a finite group $\Gamma \subset U(1)^n$.  We
shall explain the origin of $\Gamma$ and how to compute it in
\S\ref{ssec:orbifold}.  Also, note that if the critical locus of $W$ is
compact at an affine stable point then it is zero-dimensional,  and a
``nearby'' LG phase may be found by modifying the coefficients of the
monomials in $W$ so that its critical locus contracts to a single point.
In fact, compactness of the critical locus is independent of phase, as the
following lemma shows.  

\begin{lemma}
  Consider a GLSM with secondary fan $\Sigma$ and two top-dimensional
  cones $\sigma, \sigma^\prime\subset \Sigma$ and corresponding toric
  varieties $X$, $X^\prime$.  If $W$ is a function on $X$ with compact critical
  locus, the induced function $W^\prime$ on $X^\prime$ has compact critical
  locus as well.
 \label{lem:compact}
\end{lemma}

\begin{proof}
  $X$ and $X^\prime$ are related by a proper birational transformation over
  the toric variety defined by $\sigma'' = \sigma \cap \sigma'$.  The
  critical loci are related by strict transform, because $W$ and $W'$
  factor through $W''$.  Such operations do not effect the compactness or
  non-compactness of a set.
\end{proof}

\begin{remark}
  If $W \takes {\CC^N} \CC$ is a $(\CC^*)^\rho$-invariant function whose
  critical locus is compact  in some phase, its critical locus is compact
  in every phase. 
\end{remark}

\begin{theorem}
  The Herbst Criterion together with a $(\CC^*)^\rho$-invariant function on
  $\CC^N$ whose critical locus is compact after quotienting (in any/every
  phase) is equivalent to the existence of an affine Landau-Ginzburg point.
  \label{thm:main}
\end{theorem}

\begin{proof}
  Immediate, by Lemmas \ref{lem:hcasp} and \ref{lem:compact}.
\end{proof}

For convenience, let us define the cone of an $r \times n$ integer matrix as the cone over
the convex hull of its columns, thought of as $n$ elements of $\RR^r$.
\begin{definition}
  Let $T\takes {\ZZ^r} {\ZZ^n}$ and fix a basis for $\ZZ^r$.  Then
  \[
	 \text{Cone}(T) := \{\nu \in \RR^n = \ZZ^n \otimes_\ZZ \RR \:| \:
	 \nu = T(\rho) \text{ for } \rho \in \RR^r_{\geq 0}\}.
  \]
\end{definition}

\begin{remark}
  Theorem \ref{thm:main} shows that GLSMs admitting an affine Landau-Ginzburg
  phase are both extremely common and easily produced.  One may be found by
  simply choosing an arbitrary $r\times r$ integer matrix $R$ with non-zero
  determinant and selecting a finite set $S \subset \text{Cone}(-R) \cap
  \ZZ^r$. Then the matrix whose entries are $R$ and $S$, as in equation
  \eqref{eq:qform} below, satisfies the Herbst criterion.
  \label{rmk:produce}
\end{remark}

\subsection{Orbifold Structure of the Landau-Ginzburg Phase}
\label{ssec:orbifold}

Of great importance to the physics of the LG model is the finite group
$\Gamma$.  It determines the twisted sector of the model, which controls
much of the non-trivial dynamics.  As we now show, this group is inherited
from the $U(1)^r$ action in the GLSM as the stabilizer of certain coordinates
on $\CC^N$.\\

Consider a charge matrix $Q$ satisfying the Herbst Criterion, assumed for
simplicity to be a full-rank matrix.  Order the columns of $Q$ so that it
takes the form 
\begin{equation}
  Q = 
  \left(
  \begin{matrix}
    R & S\\
  \end{matrix}
  \right),
  \label{eq:qform}
\end{equation}
with $R$ an $r \times r$ integer matrix with non-zero determinant such that
\( R^{-1} \cdot Q \) is of the form given in equation \eqref{eq:hc}. \\

Choose $s$ to lie in the relative interior of  $\text{Cone}(R) \subset
\RR^r$ so that the polytope is a simplicial cone, and consider the quotient
$X$ as the algebraic quotient $(\CC^N \backslash E) / (\CC^*)^r$.  Since
the polytope is a cone, the excluded set is the union of coordinate
hyperplanes in $\CC^N$ corresponding to the half-spaces that do not define
codimension-one faces of the cone.\\

The action of the torus $(\CC^*)^r$ on \(
\spec \CC[x_1^{\pm 1}, \cdots, x_r^{\pm 1}, y_1, \cdots, y_n]
=
(\CC^N \backslash E)
\)
is then
\begin{equation}
  \begin{split}
	 \CC[x_1^{\pm 1}, \cdots, x_r^{\pm 1}, y_1, \cdots, y_n] & \ra \CC[z_1^{\pm 1}, \cdots, z_r^{\pm 1}, x_1^{\pm 1}, \cdots, x_r^{\pm 1}, y_1, \cdots, y_n] \\
	 x_i & \mapsto x_i z^{R^t_i}\\
	 y_j & \mapsto y_j z^{S^t_j}.
  \end{split}
  \label{eq:csaction}
\end{equation}

Here, we have set $z^{R^t_i} = \prod_k z_k^{R^t_{ik}}$ and $z^{S^t_j} = \prod_k
z_k^{S^t_{jk}}$ as in equation \eqref{eq:linearaction}.  Define $\Gamma \subset
(\CC^*)^r$ to be the stabilizer of the $x$'s under this action:
\begin{equation*}
  \Gamma := \spec \frac {\CC[z_1^{\pm 1}, \cdots, z_r^{\pm 1}]}{
  \langle z^{R^t_1} -1, \cdots, z^{R^t_r} -1 \rangle}  \subset (\CC^*)^r.
\end{equation*}

\begin{lemma} Consider $R^t \takes {\ZZ^r}{\ZZ^r}$ as a morphism of abelian
  groups. Then $\Gamma \cong \text{cok }R^t$ is a finite abelian group.
  \label{lem:gammaiscok}
\end{lemma}
\begin{proof}
  Note that $\Gamma \subset U(1)^r \subset (\CC^*)^r$ iff  for all $z \in
  \Gamma$, $1/z \in \Gamma$.  By inverting the relation $z^{R^t} = 1$, we
  have \( 1 = 1/z^{R_j^t} = \prod_{k=1}^r ( 1/z_k)^{R_{kj}} \)
  so $\Gamma \subset U(1)^r$.  
  Now, consider the following commutative diagram of abelian
  groups with exact rows:
  \begin{equation*}
	 \xymatrix {
	 0 \ar[r] & \ZZ^r \ar@2{-}[d] \ar[r]^i & \RR^r \ar[r]^{\text{exp}\hspace*{3mm}} & U(1)^r \ar[r]        & 0 \\
	 0 \ar[r] & \ZZ^r \ar[r]^{R^t}         & \ZZ^r \ar[r] \ar[u]_{(R^t)^{-1}}       & \text{cok }R^t \ar[r] \ar[u] & 0.
	 }
  \end{equation*}

  Let $z \in \text{cok }R^t$, and $\wt z \in \ZZ^r$ a lift of $z$.  Then,
  composing the lift with $(R^t)^{-1}$ and exp yields
  \( \wt z \mapsto \text{exp}\big (2\pi i (R^t)^{-1} \cdot \wt z\big ),\) which 
  satisfies
  \[ \text{exp} \big(2\pi i (R^t)^{-1} \wt z\big)^{R^t} = e^{2\pi i \wt
  z} = 1.\]
  Thus $\text{cok }R^t \subset \Gamma$. \\

  Let $\gamma \in \Gamma$ be a non-trivial element, and choose a lift $\wt
  \gamma \in \RR^r$.  Since $\gamma^{R^t} = 1$, $(R^t) \cdot \wt\gamma
  \subset \ZZ^r$, so $\wt \gamma$ maps to $\ZZ^r$ in the bottom row.
  Furthermore, $(R^t)^{-1} \wt \gamma$ is not in the image of $\ZZ^r$,
  since otherwise $\gamma=1$.  Thus, it lies in the cokernel and $\Gamma
  \subset \text{cok }R^t$.\\

  In both cases, lift independence follows from usual diagram chasing.
  Since $R$ is full-rank, we have that $\text{cok }R^t \cong \Gamma$ is a
  finite abelian group.
\end{proof}

It will also be important for us to know the precise form of $\Gamma$.

\begin{proposition}
  Let $D$ be the Smith normal form of $R$, $D = \text{URV}$ with U and V
  invertible over $\ZZ$, and denote its diagonal entries by $d_i$ for $1
  \leq i \leq r$.  Then 
  \[
  \Gamma \cong \ZZ_{d_1} \times \cdots \times \ZZ_{d_r}
  \cong \spec \frac { \CC[\zeta_1^{\pm 1}, \cdots, \zeta_r^{\pm 1}] } { \langle \zeta_1^{d_1} - 1, \cdots, \zeta_r^{d_r} - 1\rangle}.
  \]
\end{proposition}

\begin{proof}
  By Lemma \ref{lem:gammaiscok}, $\Gamma \cong \text{cok } R^t$.
  Transposing the expression of the Smith normal form above, we obtain a
  commutative diagram of abelian groups with exact rows:
  \begin{equation*}
	 \xymatrix {
	 0 \ar[r] & \ZZ^r \ar@2{-}[d] \ar[r]^i      & \RR^r \ar[r]^{\text{exp}\hspace*{3mm}}               & U(1)^r \ar[r]                                                    & 0 \\
	 0 \ar[r] & \ZZ^r \ar[r]^{R^t}              & \ZZ^r \ar[r] \ar[u]_{(R^t)^{-1}}                     & \Gamma \ar[r] \ar[u]                                             & 0\\
	 0 \ar[r] & \ZZ^r \ar[r]^{D^t} \ar[u]^{U^t} & \ZZ^r \ar[u]_{(V^t)^{-1}} \ar[r] \ar[d]^{(D^t)^{-1}} & \bigoplus_{a=1}^r \ZZ_{d_a} \ar@{..>}[u]_{\mid\wr} \ar[d] \ar[r] & 0\\
	 0 \ar[r] & \ZZ^r \ar@2{-}[u] \ar[r]^i      & \RR^r \ar[r]_{\text{exp}\hspace*{3mm}}               & U(1)^r \ar[r]                                                    & 0. \\
	 }
	 \hspace*{-2.0cm}
	 \xy
	 (0,3)*{}; 
	 (0,-45)*{}; 
	 **\crv{~*=<2pt>{\scriptscriptstyle .} (20,-62)&(40,-21)&(20,20)} ?(1)*\dir{>}; 
	 (33,-21)*{\scriptstyle U^t}; 
	 \endxy
  \end{equation*}
  The induced morphism between $U(1)^r$ and $U(1)^r$ obtains by first lifting
  and then composing the vertical morphisms in the center columns,
  $U^t=(R^t)^{-1}(V^t)^{-1}D^t$.  In particular, it
  is given by exponentiation as in equation
  \eqref{eq:linearaction}; for $\zeta \in U(1)^r$,
  \begin{equation}
	 \zeta \mapsto \zeta^{U^t} = (\zeta^{U_1^t}, \cdots, \zeta^{U_r^t}).
	 \label{eq:finiteiso}
  \end{equation}
  A standard diagram chase shows that the composition is independent of the
  chosen lift.  That $U^t$ is an isomorphism follows immediately from the
  fact that $U$ is invertible over the integers.  Furthermore, the induced
  action on $\bigoplus_{a=1}^r \ZZ_{d_a} \subset U(1)^r$ is the desired
  isomorphism. 
\end{proof}

The Smith normal form may be easily computed by employing, for example, the
\texttt{smithNormalForm()} command in Macaulay2\cite{M2}.  
Now that we have established that the stabilizer is a finite abelian group, 
we show that the quotient is in fact an orbifold of $\CC^n$ by this group.

\begin{theorem}
  In a Landau-Ginzburg phase the quotient $X \cong \CC^n / \Gamma$.
\end{theorem}

\begin{proof}
  Since $R$ is full rank, for all $(x,y) \in \CC^N \backslash E$, there exists
  a $z\in (\CC^*)^r$ such that $z^R x = 1$.  Thus, for any $[(x,y)] \in X$,
  \[[(x,y)] = [(1,y')].\]
  Furthermore, since $\Gamma \subset (\CC^*)^r$ is the
  stabilizer of the $x$'s, $[(1,y)] = [(1,\gamma^{S^t} y)]$ for all $\gamma \in
  \Gamma$ and $[(1,y)] \in X$.  It follows immediately that the map 
  \begin{equation*}
	 \begin{split}
		X & \ra \CC^n / \Gamma \\
		[(1,y)] & \mapsto [y]
	 \end{split}
  \end{equation*}
  is an isomorphism.
\end{proof}

\begin{remark}
  As indicated above, the action of $\Gamma$ on $\CC^n$ is specified by $S$ as in
  equation \eqref{eq:csaction} and the action of the presentation
  $\bigoplus_{a=1}^r \ZZ_{d_a}$ of $\Gamma$ is given by 
  \begin{equation*}
	 \zeta \mapsto (\zeta^{U^t})^{S^t}
  \end{equation*}
  with $\zeta^{U^t}$ as in equation \eqref{eq:finiteiso}.
\end{remark}

A natural question to ask is whether or not the affine phase is unique; that is, 
if there are two affine phases in the secondary fan, are the resulting quotients isomorphic?  The answer is affirmative, as the following theorem shows. 

\begin{theorem}
  Affine quotients are unique up to unique isomorphism.
\end{theorem}

\begin{proof}
  First, we note that the ring of regular functions for any quotient $X_s$ is
  independent of phase in which $s$ lies.  
  Rational functions on any toric variety are spanned by characters on the
  torus, and a character is regular iff it vanishes along each toric divisor.
  As the order of vanishing along a divisor is given by the entries of the map
  $A$ in Equation \ref{eq:torusseq}, the ring of regular functions depends only on $A$.
  Since the quotient $X_s$ is an affine variety whenever $s$ is an affine
  stable point, the result follows immediately.
\end{proof}

As we will see in Example \ref{ex:unique}, the finite groups may differ 
across affine phases, but they yield the same action on $\CC^n$.

\section{Examples}
\label{sec:examples}

The Herbst Criterion allows ready implementation, whether by hand or as
part of a computer program.  Typically, one restricts attention to the map
$\RR^N \rightarrow \text{Im } Q$.  In this case, Theorem \ref{thm:main}
says that 

\begin{remark}
\label{rmk:algorithm}
An affine stable point exists iff there exists a non-zero maximal minor of
$Q$ (arising say, from a submatrix $R$) such that entries of $R^{-1} Q$
away from $R$ are non-positive.  Here $R$ is inverted over $\QQ$.
\end{remark}

As in \S\ref{ssec:orbifold}, we may then order the columns of $Q$ so that
$Q = (R\; S)$, $R^{-1} S$ has non-positive rational entries, and the LG phase
is Cone$(R)$.  This follows from the changing basis for the D-terms in
\eqref{eq:dterms} to produce the form of $Q$ given in
\eqref{eq:hc}:
\begin{equation*}
  \begin{split}
	 \sum_{i=1}^N (R^{-1} \cdot Q)_i^a \abs{\phi^i}^2 &= (R^{-1}\cdot r)^a\\
	 \sum_{i=1}^N ( \mathds 1 \;\; \mathbf n)^a_i \abs{\phi^i}^2 &=  s^a.
 \end{split}
\end{equation*}
Then, $R$ provides the basis change back to the original coordinates, so that
positive $s^a$ leads to the cone over $R$.  As in the proof of Lemma
\eqref{lem:hcasp}, $s^a >0$ for all $1 \leq a \leq r$ denotes an LG point.\\ 

A program to find submatrices of a specified charge matrix that define an LG
phase is given in Appendix \ref{app:program}.\\

Physically, one would like to know which of the bosons in the theory (the Cox
coordinates) obtain VEVs in the low-energy limit and which become coordinates
in the LG phase.  As mentioned before, in an LG phase the excluded set in
$\CC^N$ is a union of coordinate hyperplanes.  In the notation above, these
coordinates correspond to the first $r$ $\phi$'s.  Assuming an appropriate
choice for superpotential, one will have that these first $r$ coordinates
obtain VEVs while the remaining $n$ become coordinates in the LG phase.\\

\begin{example}[Example \ref{ex:canonicalpm}, redux]
  Consider a GLSM with charge matrix $Q=(1,1,-2)$ and superpotential $W =
  \phi_0 ( \phi_1^2 + \phi^2_2 + \phi_1 \phi_2 )$. In the geometric phase the
  quotient variety is the canonical bundle of $\PP^1$.  The D-term in this
  model is 
  \[ \abs{\phi_1}^2 + \abs{\phi_2}^2 - 2 \abs{\phi_0}^2 = r.\]
  As before, one takes $R=(-2)$, so that $\mathbf n = (-\half, -\half)$, 
  $s = -\half r$, and the image of the polytope in $\RR^3$ is
  \begin{align*}
	 \abs{\phi_0}^2 = s+ \half \Big ( \abs{\phi_1}^2 + \abs{\phi_2}^2 \Big ),
  \end{align*}

  Since the excluded set for positive $s$ is the coordinate hyperplane
  $\{\phi_0=0\}$, critical points of the superpotential occur only at $\phi_1 =
  \phi_2 = 0$.  Thus $\phi_0$ obtains a VEV, which may be chosen as
  $\langle\phi_0\rangle = \sqrt{s}$, while $\phi_1$ and $\phi_2$ become
  coordinates in the LG phase.  It is furthermore clear that the Smith normal
  form of $R$ is $(2)$, so that the finite group is $\ZZ_2$ acting as $(\phi_1,
  \phi_2) \mapsto (-\phi_1, -\phi_2)$.
  \label{ex:lgphase}
\end{example}

\begin{example}[Example \ref{example:canonicalrwp}, redux]
  In this model, we found that $Q$ in equation \eqref{eq:rwp4cm} may be
  written as $Q = (R\quad S)$ with
  \[ R = 
  \left(
  \begin{matrix}
	 1 & -4 \\
	 -2 & 0 
  \end{matrix}
  \right) \quad\text{ and }\quad 
  S = \left(\begin{matrix} 0&0&1&1&1\\ 1&1&0&0&0\end{matrix}\right),
  \]
  so the LG phase is given by Cone$(R)$ for appropriate choice of
  superpotential.  See Figure \ref{fig:lgphase}, and compare with Figure 2 of
  \cite{Morrison:1994fr} (up to quantum corrections).  It is easy to check that
  the Smith normal form of $R$ is 
  \[
  D = \text{U R V} = 
  \left(
  \begin{matrix}
	 8 & 0 \\
	 0 & 1 
  \end{matrix}
  \right)
  =
  \left(
  \begin{matrix}
	 2 & 1 \\
	 1 & 1 
  \end{matrix}
  \right)
  \left(
  \begin{matrix}
	 1 & -4 \\
	 -2 & 0 
  \end{matrix}
  \right)
  \left(
  \begin{matrix}
	 4 & -1 \\
	 -1 & 0 
  \end{matrix}
  \right),
  \]
  so that $\Gamma \cong \ZZ_8 \times \ZZ_1 \subset U(1)^2$ and $X =
  \CC^5/\ZZ_8$.  To find the action of $\ZZ_8$ on $\CC^5$, let $(\zeta_1,
  \zeta_2)$ be generators of $\ZZ_8 \times \ZZ_1$ so that 
  \[
  \begin{split} 
	 \zeta &\mapsto \zeta^{U^t_j}\\
	 (\zeta_1, \zeta_2 = 1) &\mapsto ( \zeta_1^2 \zeta_2,
	 \zeta_1 \zeta_2) = (\zeta_1^2, \zeta_1),
  \end{split}
  \]
  and the action on
  coordinates is determined by S as:
  \[
  \xymatrix{
  (\phi_1, \phi_2, \phi_3, \phi_4, \phi_5 )
  \ar@{|->}[r]^{\hspace*{-2em}(\zeta^{U^t})^{S^t}} & 
  (\zeta_1 \phi_1 , \zeta_1 \phi_2, \zeta_1^2 \phi_3, \zeta_1^2 \phi_4, \zeta_1^2 \phi_5 ).
  }
  \]
  \label{ex:crwp2}
\end{example}

\begin{figure}[t]
  \centering
  \includegraphics{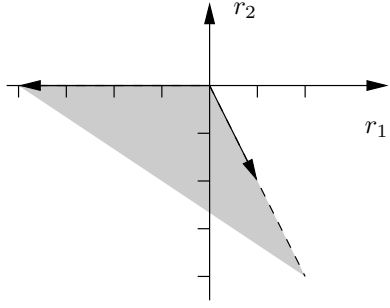}
  \caption{The Landau-Ginzburg phase for Example \ref{ex:crwp2}.}
  \label{fig:lgphase}
\end{figure}

\begin{example}
  Consider the charge matrix 
  \begin{equation}
	 \label{eq:twoLG}
	 Q = \left( \begin{matrix} 
		0 & 1 & 1 & 1 & 1 & -4 \\
		1 & 0 & 0 & 0 & -2 & 0
	 \end{matrix} \right). 
  \end{equation}
  There are two cones of the secondary fan containing affine stable points:
  \[
  \text{Cone} ( (0, 1), (-4, 0)) \qquad \text{ and } \qquad
  \text{Cone} ( (1, -2), (-4, 0)).
  \]
  One can easily check that the finite groups are $\ZZ_4$ and $\ZZ_8$, respectively, but 
  both have $\ZZ_4$ actions on $\CC^6$.
  \label{ex:unique}
\end{example}

\appendix
\section{Algorithm Implementation}
\label{app:program}

The following Mathematica\cite{Mathematica7} program takes a full-rank
matrix of charges $Q$, and returns a list of the Landau-Ginzburg phases.
Each phase is presented as the submatrix $R$ of $Q$ as in Remark
\ref{rmk:algorithm} along with the column numbers of $Q$ from which $R$ was
obtained.  \\

\begin{verbatim}
InvertibleSubs[Q_List] := Module[{r,out,sub,i,j,sets,Qsub},
  out = {};
  r = MatrixRank[Q];
  sets = Subsets[Table[i, {i, 1, Length[Q[[1]]]}], {r}];
  For[i = 1, i <= Length[sets], i++,
   sub = Transpose[(Transpose[Q][[#]] & /@ sets[[i]])];
   If[Det[sub] != 0,
    Qsub = Q;
    For[j = 1, j <= Length[sets[[i]]], j++,
     Qsub = Transpose[Drop[Transpose[Qsub],
	     {Sort[sets[[i]], Greater][[j]]}]]
    ];
    Qsub = Inverse[sub].Qsub;
    If[Plus @@ (If[# > 0, 1, 0] & /@ Flatten[Qsub]) == 0,
     AppendTo[out, {sub, sets[[i]]}];
    ]
   ]
  ];
  Return[out];
]
\end{verbatim}

\bibliography{all}
\end{document}